\documentclass{amsart}
\usepackage[utf8]{inputenc}
\usepackage{graphicx}
\usepackage{tikz-cd}
\usepackage{mathrsfs}
\usepackage{amsfonts}
\usepackage{amsthm}
\usepackage{amsmath}
\usepackage{amssymb}
\usepackage{mathtools}
\usepackage{extarrows}
\usepackage{interval}
\usepackage{bbm}
\usepackage[english]{babel}
\usepackage{romannum}
\usepackage{setspace}
\usepackage{listings}
\usepackage{pdfpages}
\usepackage{hyperref}
\usepackage{enumitem}

\newtheorem{theorem}{Theorem}[section]
\newtheorem{lemma}[theorem]{Lemma}
\newtheorem{corollary}[theorem]{Corollary}
\newtheorem{proposition}[theorem]{Proposition}

\theoremstyle{definition}
\newtheorem{definition}[theorem]{Definition}

\numberwithin{equation}{section}
\newcommand*{\sheafhom}{\mathcal{H}\kern -.5pt om}
\newcommand*{\sheafext}{\mathcal{E}\kern -.5pt xt}
\newcommand*{\Cohdim}{\operatorname{cd}}

\title{Orlov's theorem over a quasiexcellent ring}
\author{Fei Peng}
\email{pengf2@student.unimelb.edu.au}
\address{School of Mathematics \& Statistics, The University of Melbourne, Parkville, VIC, 3010, Australia}
\date{\today}

\begin{document}

\begin{abstract}
    Following the approach of Kawamata and Canonaco--Stellari, we establish Orlov's representability theorem for smooth tame Deligne--Mumford stacks with projective coarse moduli spaces over a quasiexcellent ring of finite Krull dimension. This generalizes a previous result of Canonaco--Stellari for smooth projective varieties over a field.
\end{abstract}

\maketitle
\allowdisplaybreaks
\pagenumbering{arabic}

\section{Introduction}\label{I}

Let $X$ and $Y$ be smooth projective varieties over a field $k$, and let $F\colon D_{\operatorname{Coh}}^{b}(X)\to D_{\operatorname{Coh}}^{b}(Y)$ be a $k$-linear exact functor. We say $F$ that is a Fourier--Mukai transform if there exists a complex $K\in D_{\operatorname{Coh}}^{b}(X\times_{k}Y)$ such that $F\cong Rp_{2,*}(Lp_{1}^{*}(-)\otimes_{\mathcal{O}_{X\times_{k}Y}}^{\mathbb{L}}K)$, where $p\colon X\times_{k}Y\to Y$ and $q\colon X\times_{k}Y\to X$ are the canonical projections. A fundamental theorem of Orlov \cite[Theorem\ 2.2]{Orl97} states that if $F$ is fully faithful, then it is a Fourier-Mukai transform. Over the years, Orlov's theorem has been improved in several directions \cite{Kaw04, CS07, Bal09, LO10, CS14, Ola24, Pen24}. However, all the results listed above are stated over a field.

The purpose of this short note is to establish, with some mild assumptions, Orlov's theorem for tame algebraic stacks with projective coarse moduli spaces over a quasiexcellent ring. This partially answers a question of Neeman.

\begin{theorem}\label{1.1}
Let $R$ be a quasiexcellent ring of finite Krull dimension. Let $\mathcal{X},\mathcal{Y}$ be smooth proper tame algebraic stacks over $R$. Let $\pi\colon\mathcal{X}\to X$ be the associated coarse moduli space. Let $F\colon D_{\operatorname{\operatorname{Coh}}}^{b}(\mathcal{X})\to D_{\operatorname{Coh}}^{b}(\mathcal{Y})$ be an $R$-linear exact functor such that$$\operatorname{Hom}_{\mathcal{O}_{\mathcal{Y}}}(F(\mathcal{A}),F(\mathcal{B})[i])=0$$for all $\mathcal{A},\mathcal{B}\in\operatorname{Coh}(\mathcal{X})$ and $i<0$. If $\mathcal{X}$ has the resolution property and $X$ admits an ample line bundle, then there exists a bounded complex $K\in D_{\operatorname{Coh}}^{b}(\mathcal{X}\times_{R}\mathcal{Y})$ such that $F\cong\Phi_{K}$, where $\Phi_{K}=Rp_{*}(Lq^{*}(-)\otimes_{\mathcal{O}_{X\times_{R}Y}}^{\mathbb{L}}K)$ and $p\colon\mathcal{X}\times_{R}\mathcal{Y}\to\mathcal{Y}$ and $q\colon\mathcal{X}\times_{R}\mathcal{Y}\to\mathcal{X}$ are the canonical projections. Moreover, the complex $K$ is unique up to quasi-isomorphism.
\end{theorem}

Theorem \ref{1.1} is even new for smooth projective varieties over $\mathbb{Z}$. Our method is a natural extension of the approach in \cite{Kaw04} and \cite{CS07}. One could obtain a variant of Theorem \ref{1.1} for twisted sheaves following \cite{CS07} as well. Quasiexcellence is necessary for us to show that $F$ admits a left adjoint and is bounded. Both statements are consequences of the strong generation of derived categories of coherent sheaves on tame stacks \cite[Theorem\ B.1]{HP24}. See Lemma \ref{5.1} for the details.
 
It is known that every smooth separated Deligne--Mumford stack with quasiprojective coarse moduli space over an equicharacteristic base has the resolution property \cite[Theorem\ 6.3]{BHM21}. In the appendix, we record a mixed characteristic variant of this result in the generically tame case, which we expect to be folklore among experts. As a consequence, we see that the assumption that $\mathcal{X}$ has the resolution property in Theorem \ref{1.1} would be superfluous if $\mathcal{X}$ is Deligne--Mumford.

\subsection*{Acknowledgement}
This project is supported by a Melbourne Research Scholarship. It is my pleasure to thank Amnon Neeman for asking whether Theorem \ref{1.1} holds and for his valuable insights and comments on an earlier draft of this article. I am also very grateful to my supervisor, Jack Hall, for his guidance, encouragement, and support throughout this project. I would also like to thank Oliver Li and Adam Monteleone for many helpful discussions.

%Parts of this article were written during the International Conference on Triangulated Categories and Related Topics 2024 at Northeast Normal University in Changchun, China. I would like to thank the organizers and the hosting institution for their hospitality and support during my visit to Changchun.

\section{Resolutions of the diagonal}\label{2}

In this section, we produce resolutions of the diagonal for algebraic stacks with the resolution property. Let $\mathcal{X}$ be a quasicompact and quasiseparated algebraic stack. We say $\mathcal{X}$ has the \textit{resolution property} if every quasicoherent sheaf of finite type on $\mathcal{X}$ admits a surjective map from some vector bundle on $\mathcal{X}$. For example, every quasiprojective scheme has the resolution property. In general, it is very difficult to determine whether an algebraic stack has the resolution property. It is closely related to the global quotient structure of algebraic stacks and the surjectivity of the Brauer map. We refer to \cite{EHKV01}, \cite{tot04}, and \cite{gro17} for detailed discussions on their connections. The following result is a non-derived variant of \cite[Proposition\ 2.24]{BZFN10}, which we expect to be well-known.

\begin{proposition}\label{2.2}
Let $R$ be a ring. Let $\mathcal{X}$ and $\mathcal{Y}$ be algebraic stacks of finite type over $R$ with affine diagonal. If $\mathcal{X}$ and $\mathcal{Y}$ both have the resolution property, then any coherent sheaf $\mathcal{F}$ on $\mathcal{X}\times_{R}\mathcal{Y}$ admits a resolution of the following form
\begin{equation}
    \begin{tikzcd}
    \cdots\arrow[r] & \mathcal{E}_{i}\boxtimes\mathcal{G}_{i}\arrow[r] & \cdots\arrow[r] & \mathcal{E}_{0}\boxtimes\mathcal{G}_{0}\arrow[r] & \mathcal{F}\arrow[r] & 0,
    \end{tikzcd}
\end{equation}
where $\mathcal{E}_{i}$ and $\mathcal{G}_{i}$ are vector bundles on $\mathcal{X}$ and $\mathcal{Y}$ respectively for all $i\geq 0$.
\end{proposition}
\begin{proof}
We need to show if $\operatorname{Hom}_{\mathcal{O}_{\mathcal{X}\times_{R}\mathcal{Y}}}(\mathcal{E}\boxtimes\mathcal{G},\mathcal{F})\cong0$ for every pair of vector bundles $\mathcal{E}$ on $\mathcal{X}$ and $\mathcal{G}$ on $\mathcal{Y}$, then $\mathcal{F}\cong0$. By adjunction, we have
\begin{align*}
    0=&\operatorname{Hom}_{\mathcal{O}_{\mathcal{X}\times_{R}\mathcal{Y}}}(\mathcal{E}\boxtimes\mathcal{G},\mathcal{F})\\
    \cong &\operatorname{Hom}_{\mathcal{O}_{\mathcal{X}\times_{R}\mathcal{Y}}}(q^{*}\mathcal{E},\sheafhom_{\mathcal{O}_{\mathcal{X}\times_{R}\mathcal{Y}}}(p^{*}\mathcal{G},\mathcal{F}))\\
    \cong & \operatorname{Hom}_{\mathcal{O}_{\mathcal{X}}}(\mathcal{E},q_{*}\sheafhom_{\mathcal{O}_{\mathcal{X}\times_{R}\mathcal{Y}}}(p^{*}\mathcal{G},\mathcal{F})).
\end{align*}
Since this is true for every $\mathcal{E}$ on $\mathcal{X}$, we have $q_{*}\sheafhom_{\mathcal{O}_{\mathcal{X}\times_{R}\mathcal{Y}}}(p^{*}\mathcal{E},\mathcal{F})=0$.

By assumption, $\mathcal{X}$ admits a smooth cover $\pi\colon\operatorname{Spec}(A)\to\mathcal{X}$ of $\mathcal{X}$ by an affine scheme. Therefore we have the following diagram
\begin{equation*}
    \begin{tikzcd}
    \operatorname{Spec}(A)\times_{R}\mathcal{Y}\arrow[r,"q_{A}"]\arrow[d,"\pi_{\mathcal{Y}}"]\arrow[dd,bend right=60,"\tilde{p}"] & \operatorname{Spec}(A)\arrow[d,"\pi"]\\
    \mathcal{X}\times_{R}\mathcal{Y}\arrow[r,"q"]\arrow[d,"p"] & \mathcal{X}\arrow[d]\\
    \mathcal{Y}\arrow[r] & \operatorname{Spec}(R)
    \end{tikzcd}
\end{equation*}
where the upper and lower square are both cartesian and $\tilde{p}=p\circ\pi_{\mathcal{Y}}$. By construction, the composition $\tilde{p}$ is affine. Since $q_{*}\sheafhom_{\mathcal{O}_{\mathcal{X}\times_{R}\mathcal{Y}}}(p^{*}\mathcal{E},\mathcal{F})=0$, we have
\begin{align*}
    0\cong& \Gamma(\operatorname{Spec}(A),\pi^{*}q_{*}\sheafhom_{\mathcal{O}_{\mathcal{X}\times_{R}\mathcal{Y}}}(p^{*}\mathcal{E},\mathcal{F}))\\
    \cong& \Gamma(\operatorname{Spec}(A), q_{A,*}\pi_{\mathcal{Y}}^{*}\sheafhom_{\mathcal{O}_{\mathcal{X}\times_{R}\mathcal{Y}}}(p^{*}\mathcal{E},\mathcal{F}))\\
    \cong& \Gamma(\operatorname{Spec}(A)\times_{R}\mathcal{Y},\sheafhom_{\mathcal{O}_{\operatorname{Spec}(A)\times_{R}\mathcal{Y}}}(\tilde{p}^{*}\mathcal{E},\pi_{\mathcal{Y}}^{*}\mathcal{F}))\\
    \cong& \Gamma(\mathcal{Y},\tilde{p}_{*}\sheafhom_{\mathcal{O}_{\operatorname{Spec}(A)\times_{R}\mathcal{Y}}}(\tilde{p}^{*}\mathcal{E},\pi_{\mathcal{Y}}^{*}\mathcal{F}))\\
    \cong& \Gamma(\mathcal{Y},\sheafhom_{\mathcal{O}_{\mathcal{Y}}}(\mathcal{E},\tilde{p}_{*}\pi_{\mathcal{Y}}^{*}\mathcal{F}))\\
    \cong& \operatorname{Hom}_{\mathcal{O}_{\mathcal{Y}}}(\mathcal{E},\tilde{p}_{*}\pi_{\mathcal{Y}}^{*}\mathcal{F}).
\end{align*}
It follows that $\tilde{p}_{*}\pi_{\mathcal{Y}}^{*}\mathcal{F}\cong 0$. Since $\tilde{p}$ is affine, $\tilde{p}_{*}$ is conservative, and thus $\pi_{\mathcal{Y}}^{*}\mathcal{F}\cong 0$. But $\pi_{\mathcal{Y}}$ is faithfully flat, so we must have $\mathcal{F}=0$ as desired.
\end{proof}
In the special case where $\mathcal{X}$ and $\mathcal{Y}$ are projective varieties, Proposition \ref{2.2} implies that every box product of ample line bundles on $\mathcal{X}$ and $\mathcal{Y}$ is again ample on $\mathcal{X}\times_{R}\mathcal{Y}$, which is usually proved using the Segre embedding. An immediate consequence of Proposition \ref{2.2} is the desired resolution of the diagonal for $\mathcal{X}$.
\begin{corollary}\label{2.3}
Let $\mathcal{X}$ be a separated algebraic stack of finite type over $R$ with the resolution property. Then the sheaf $\mathcal{O}_{\Delta_{\mathcal{X}}}$ on $\mathcal{X}\times_{k}\mathcal{X}$ admits a resolution of the following form
\begin{equation}\label{eq2.3}
    \begin{tikzcd}
    \cdots\arrow[r] & \mathcal{E}_{i}\boxtimes\mathcal{G}_{i}\arrow[r] & \cdots\arrow[r] & \mathcal{E}_{0}\boxtimes\mathcal{G}_{0}\arrow[r] & \mathcal{O}_{\Delta_{\mathcal{X}}}\arrow[r] & 0,
    \end{tikzcd}
\end{equation}
where $\mathcal{E}_{i}$ and $\mathcal{G}_{i}$ are finite locally free sheaves on $\mathcal{X}$ for all $i\geq 0$.
\end{corollary}

\section{Generating sheaves and ample sequences}\label{3}

In this section, we produce an ample sequence in the category of coherent sheaves $\operatorname{Coh}(\mathcal{X})$ for a smooth separated tame irreducible algebraic stack $\mathcal{X}$ over a noetherian ring $R$ with the resolution property and a projective coarse moduli space. We first remind the reader of the notion of an ample sequence, which is due to Orlov.
\begin{definition}[{\cite[Definition\ 2.72]{Huy06}}]\label{3.1}
Let $\mathcal{A}$ be an abelian category. A sequence of objects $\{\mathcal{L}_{i}\}_{i\in\mathbb{Z}}$ in $\mathcal{A}$ is {\it ample} if for every object $\mathcal{F}\in\mathcal{A}$, there exists an integer $N$ such that for all $n>N$:
\begin{enumerate}[topsep=0pt,noitemsep,label=\normalfont(\arabic*)]
    \item\label{3.1.1} the canonical map $\operatorname{Hom}(\mathcal{L}_{n},\mathcal{F})\otimes\mathcal{L}_{n}\to \mathcal{F}$ is surjective,
    \item\label{3.1.2} $\operatorname{Hom}(\mathcal{L}_{n},\mathcal{F}[i])=0$ for all $i\neq 0$, and
    \item\label{3.1.3} $\operatorname{Hom}(\mathcal{F},\mathcal{L}_{n})=0$.
\end{enumerate}
\end{definition}

For example, let $X$ be a smooth projective variety over a field. Let $\mathcal{L}$ be an ample line bundle on $X$. The sequence $\{\mathcal{L}^{\otimes -n}\}_{n\in\mathbb{Z}}$ is ample in $\operatorname{Coh}(X)$. Indeed, the first two conditions follow directly from ampleness. The last one follows from the second by Serre duality. Ample sequences are particularly useful for extending natural isomorphisms of functors. See \cite[Proposition\ 3.7]{CS07} for the details.

Unfortunately, it is impossible to carry the notion of ampleness over directly to stacks. To remedy the situation, we appeal to the following notion of a generating vector bundle introduced by Olsson and Starr.

\begin{definition}[{\cite[Definition\ 5.1]{OS03}}]\label{3.2}
Let $\mathcal{X}$ be a tame algebraic stack with coarse moduli space $\pi\colon\mathcal{X}\to X$. A vector bundle $\mathcal{V}$ on $\mathcal{X}$ is {\it generating} if the natural map$$\pi^{*}\pi_{*}\sheafhom_{\mathcal{O_{X}}}(\mathcal{V},\mathcal{F})\otimes_{\mathcal{O_{X}}}\mathcal{V}\to\mathcal{F}$$is surjective.
\end{definition}

Consider the classifying stack $B_{k}G$ over a field $k$ where $G$ is a finite group such that $|G|$ is coprime to the characteristic of $k$. We see that $G$ is linearly reductive over $k$, and thus $B_{k}G$ is a tame stack. Quasicoherent sheaves on $B_{k}G$ are nothing but representations of $G$ over $k$. Let $\mathcal{V}$ be the locally free sheaves corresponding to the regular representation of $G$ over $k$. Then $\mathcal{V}$ is a generating vector bundle on $B_{k}G$. This follows from the fact that every irreducible representation of $G$ is a direct summand of the regular representation. 

Recall that $B_{k}G$ is a global quotient stack. It was shown that every separated tame quotient Deligne--Mumford stack admits a generating vector bundle \cite[Theorem\ 5.5]{OS03}. We claim that this holds for separated tame quotient algebraic stacks as well. Let $\mathcal{X}$ be such a stack, and let $X$ be the associated coarse moduli space. By \cite[Theorem\ 3.2]{AOV08}, there exists an \'etale cover $X^\prime\to X$ such that $\mathcal{X}\times_{X}X^\prime\cong[Y/G]$ where $Y$ is a scheme that is finite over $X$ and $G$ is a finite linearly reductive group. The statement then follows from the proof of \cite[Theorem\ 5.5]{OS03}. In particular, every smooth tame Deligne--Mumford stack with projective coarse moduli space over a noetherian scheme admits a generating vector bundle by Corollary \ref{A.2}. Combining the surjective map in Definition \ref{3.2} and the global generation property of ample line bundles on projective schemes, we obtain the following lemma due to Kresch \cite[Section\ 5.2]{kre09}.

\begin{lemma}\label{3.3}
Let $R$ be a noetherian ring. Let $\mathcal{X}$ be a smooth separated tame algebraic stack over $R$ with coarse moduli space $\pi\colon\mathcal{X}\to X$. If $\mathcal{X}$ admits a generating vector bundle $\mathcal{V}$ and $X$ admits an ample line bundle $L$, then the sum of the global morphisms of $\mathcal{V}\otimes\pi^{*}L^{\otimes-n}$ into any quasicoherent sheaf $\mathcal{F}$ on $\mathcal{X}$ generates $\mathcal{F}$.
\end{lemma}

\begin{proof}
Since $\mathcal{V}$ is generating, there exists a surjective map$$\pi^{*}\pi_{*}\sheafhom_{\mathcal{O_{X}}}(\mathcal{V},\mathcal{F})\otimes\mathcal{V}\twoheadrightarrow\mathcal{F}$$ for every quasicoherent sheaf $\mathcal{F}$ on $\mathcal{X}$. Note that $\pi_{*}\sheafhom_{\mathcal{O_{X}}}(\mathcal{V},\mathcal{F})$ is a quasicoherent sheaf on $X$. By ampleness, we have a surjective map$$\bigoplus_{n\geq 1}\Gamma(X,\pi_{*}\sheafhom_{\mathcal{O_{X}}}(\mathcal{V},\mathcal{F})\otimes_{\mathcal{O}_{X}}L^{\otimes n})\otimes_{R}L^{\otimes -n}\twoheadrightarrow\pi_{*}\sheafhom_{\mathcal{O_{X}}}(\mathcal{V},\mathcal{F}).$$Since $\pi^{*}$ and $-\otimes_{\mathcal{O_{X}}}\mathcal{V}$ are right exact and $\mathcal{V}$ is generating, we have a surjective map$$\bigoplus_{n\geq 1}\Gamma(X,\pi_{*}\sheafhom_{\mathcal{O}_{X}}(\mathcal{V},\mathcal{F})\otimes_{\mathcal{O_{X}}}L^{\otimes n})\otimes_{R}(\mathcal{V}\otimes_{\mathcal{O_{X}}}\pi^{*}L^{\otimes -n})\twoheadrightarrow\mathcal{F}.$$Since $\pi\colon\mathcal{X}\to X$ is a tame moduli space, we have
\begin{align*}
    &\bigoplus_{n\geq 1}\Gamma(X,\pi_{*}\sheafhom_{\mathcal{O_{X}}}(\mathcal{V},\mathcal{F})\otimes_{\mathcal{O}_{X}}L^{\otimes n})\otimes_{R}(\mathcal{V}\otimes_{\mathcal{O_{X}}}\pi^{*}L^{\otimes -n})&\\
    \cong & \bigoplus_{n\geq 1}\Gamma(X,\pi_{*}(\sheafhom_{\mathcal{O_{X}}}(\mathcal{V},\mathcal{F})\otimes_{\mathcal{O_{X}}}\pi^{*}L^{\otimes n}))\otimes_{R}(\mathcal{V}\otimes_{\mathcal{O_{X}}}\pi^{*}L^{\otimes -n})&\\
    \cong & \bigoplus_{n\geq 1}\Gamma(\mathcal{X},\sheafhom_{\mathcal{O_{X}}}(\mathcal{V},\mathcal{F})\otimes_{\mathcal{O_{X}}}\pi^{*}L^{\otimes n})\otimes_{R}(\mathcal{V}\otimes_{\mathcal{O_{X}}}\pi^{*}L^{\otimes -n})&\\
    \cong & \bigoplus_{n\geq 1}\operatorname{Hom}_{\mathcal{O_{X}}}(\mathcal{V}\otimes_{\mathcal{O_{X}}}\pi^{*}L^{\otimes -n},\mathcal{F})\otimes_{R}(\mathcal{V}\otimes_{\mathcal{O_{X}}}\pi^{*}L^{\otimes -n}).
\end{align*}
This finishes the proof.
\end{proof}

We wish to show that the sequence $S=\{\mathcal{V}\otimes_{\mathcal{O_{X}}}\pi^{*}L^{\otimes -n}\}_{n\in\mathbb{Z}}$ is ample in $\operatorname{Coh}(\mathcal{X})$ if $\mathcal{X}$ is proper. Lemma \ref{3.3} tells us that $S$ satisfies condition \ref{3.1.1} in Definition \ref{3.1}. We now verify the other two conditions. 

\begin{lemma}\label{3.4}
Let $R$ be a noetherian ring. Let $\mathcal{X}$ be a smooth separated tame algebraic stack over $R$ with coarse moduli space $\pi\colon\mathcal{X}\to X$. Let $\mathcal{V}$ be a vector bundle on $\mathcal{X}$ and $L$ be an ample line bundle on $X$. Let $\mathcal{F}$ be a coherent sheaf on $\mathcal{X}$. If $X$ is proper, then there exists $N\in\mathbb{N}$ such that$$\operatorname{Ext}_{\mathcal{O_{X}}}^{i}(\mathcal{V}\otimes_{\mathcal{O_{X}}}\pi^{*}L^{\otimes -n},\mathcal{F})=0$$for all $i>0$ and $n>N$.
\end{lemma}

\begin{proof}
By adjunction, we have
\begin{align*}
    \operatorname{Ext}_{\mathcal{O_{X}}}^{i}(\mathcal{V}\otimes\pi^{*}L^{\otimes -n},\mathcal{F})\cong&\operatorname{Ext}_{\mathcal{O_{X}}}^{i}(\pi^{*}L^{\otimes -n},\sheafhom_{\mathcal{O_{X}}}(\mathcal{V},\mathcal{F}))\\
    \cong&\operatorname{Ext}_{\mathcal{O_{X}}}^{i}(L^{\otimes -n},\pi_{*}\sheafhom_{\mathcal{O_{X}}}(\mathcal{V},\mathcal{F}))\ \ (\pi_{*}\ \text{is exact})\\
    \cong&H^{i}(X,\pi_{*}\sheafhom_{\mathcal{O_{X}}}(\mathcal{V},\mathcal{F})\otimes_{\mathcal{O}_{X}}L^{\otimes n})
\end{align*}
Since $\pi$ is proper, $\pi_{*}\sheafhom_{\mathcal{O_{X}}}(\mathcal{V},\mathcal{F})$ is coherent. Applying Serre vanishing to $X$ and $\pi_{*}\sheafhom_{\mathcal{O_{X}}}(\mathcal{V},\mathcal{F})$ finishes the proof.
\end{proof}

It remains to verify \ref{3.1.3} for $S$. We do so by reducing to the case where $R$ is a field using cohomology and base change for algebraic stacks.

\begin{lemma}\label{3.5}
Let $R$ be a noetherian ring. Let $\mathcal{X}$ be a smooth separated irreducible tame algebraic stack over $R$ with coarse moduli space $\pi\colon\mathcal{X}\to X$. Let $\mathcal{V}$ be a generating vector bundle on $\mathcal{X}$ and $L$ be an ample line bundle on $X$. Let $\mathcal{F}$ be a coherent sheaf on $\mathcal{X}$. If $X$ is proper, then there exists an integer $N\in\mathbb{N}$ such that$$\operatorname{Hom}_{\mathcal{O}_{\mathcal{X}}}(\mathcal{F},\mathcal{V}\otimes_{\mathcal{O_{X}}}\pi^{*}L^{\otimes -n})=0$$for all $n>N$.
\end{lemma}

\begin{proof}
Suppose for any closed point $s$ of $\operatorname{Spec}(R)$, there exists an integer $N_{s}$ such that $\operatorname{Hom}_{\mathcal{O}_{\mathcal{X}_{s}}}(\mathcal{F}_{s},\mathcal{V}_{s}\otimes_{\mathcal{O}_{\mathcal{X}_{s}}}\pi_{s}^{*}L_{s}^{\otimes -n})=0$ for all $n>N_{s}$. We first claim that $\operatorname{Hom}_{\mathcal{O}_{\mathcal{X}_{s}}}(L i_{s}^{*}\mathcal{F},\mathcal{V}_{s}\otimes_{\mathcal{O}_{\mathcal{X}_{s}}}\pi_{s}^{*}L_{s}^{\otimes -n})=0$ for all $n>N_{s}$, where $i_{s}\colon\mathcal{X}_{s}\to\mathcal{X}$ is the closed immersion of the fibre of $\mathcal{X}\to\operatorname{Spec}(R)$ over $s$, and $L i_{s}^{*}\mathcal{F}$ is the derived pullback. To see this, we note that for every coherent sheaf $\mathcal{F}$ on $\mathcal{X}$, there is a distinguished triangle$$\tau^{\leq -1}L i_{s}^{*}\mathcal{F}\longrightarrow L i_{s}^{*}\mathcal{F}\longrightarrow\mathcal{H}^{0}(L i_{s}^{*}\mathcal{F}_{s})\simeq\mathcal{F}_{s}\longrightarrow\tau^{\leq -1}L i_{s}^{*}\mathcal{F}[1].$$For every vector bundle $\mathcal{N}$ on $\mathcal{X}$, applying $\operatorname{Hom}_{\mathcal{O}_{\mathcal{X}_{s}}}(-,\mathcal{N}_{s})$ to the triangle above yields the following long exact sequence$$\cdots\to\operatorname{Hom}_{\mathcal{O}_{\mathcal{X}_{s}}}(\mathcal{F}_{s},\mathcal{N}_{s})\to\operatorname{Hom}_{\mathcal{O}_{\mathcal{X}_{s}}}(L i_{s}^{*}\mathcal{F},\mathcal{N}_{s})\to\operatorname{Hom}_{\mathcal{O}_{\mathcal{X}_{s}}}(\tau^{\leq -1}L i_{s}^{*}\mathcal{F},\mathcal{N}_{s})\to\cdots.$$Since $\tau^{\leq -1}L i_{s}^{*}\mathcal{F}\in D_{\operatorname{Coh}}^{\leq -1}(\mathcal{X}_{s})$ while $\mathcal{N}_{s}$ is concentrated in degree 0, we see that $\operatorname{Hom}_{\mathcal{O}_{\mathcal{X}_{s}}}(\tau^{\leq -1}L i_{s}^{*}\mathcal{F},\mathcal{N}_{s})=0$, and the map $\operatorname{Hom}_{\mathcal{O}_{\mathcal{X}_{s}}}(\mathcal{F}_{s},\mathcal{N}_{s})\to\operatorname{Hom}_{\mathcal{O}_{\mathcal{X}_{s}}}(L i_{s}^{*}\mathcal{F},\mathcal{N}_{s})$ is surjective. Setting $\mathcal{N}=\mathcal{V}\otimes_{\mathcal{O_{X}}}\pi^{*}L^{\otimes -n}$, we have $\operatorname{Hom}_{\mathcal{O}_{\mathcal{X}_{s}}}(\mathcal{F}_{s},\mathcal{V}_{s}\otimes_{\mathcal{O}_{\mathcal{X}_{s}}}\pi_{s}^{*}L_{s}^{\otimes -n})=0$ for all $n>N_{s}$. It follows that $\operatorname{Hom}_{\mathcal{O}_{\mathcal{X}_{s}}}(L i_{s}^{*}\mathcal{F},\mathcal{V}_{s}\otimes_{\mathcal{O}_{\mathcal{X}_{s}}}\pi_{s}^{*}L_{s}^{\otimes -n})=0$ for all $n>N_{s}$ as desired. By \cite[Theorem\ A]{Hal14}, there exists an open neighborhood $U_{s}\subset|\operatorname{Spec}(R)|$ such that $\operatorname{Hom}_{\mathcal{O}_{\mathcal{X}_{U_{s}}}}(\mathcal{F}_{X_{U_{s}}},\mathcal{V}_{s}\otimes_{\mathcal{O}_{\mathcal{X}_{U_{s}}}}\pi_{\mathcal{X}_{U_{s}}}^{*}\mathcal{L}_{\mathcal{X}_{U_{s}}}^{\otimes -n})=0$. Since $R$ is noetherian, we can cover $\operatorname{Spec}(R)$ with finitely many such open neighborhoods. Setting $N$ as the maximum among the $N_{s}$s obtained from the finite covering above would finish the proof. Therefore, we may assume $R$ is a field. In this case, we have
\begin{align*}
    \operatorname{Hom}_{\mathcal{O_{X}}}(\mathcal{F},\mathcal{V}\otimes_{\mathcal{O_{X}}}\pi^{*}L^{\otimes -n})\cong&\Gamma(\mathcal{X},\sheafhom_{\mathcal{O_{X}}}(\mathcal{F},\mathcal{V}\otimes_{\mathcal{O_{X}}}\pi^{*}L^{\otimes -n}))\\
    \cong&\Gamma(\mathcal{X},\sheafhom_{\mathcal{O_{X}}}(\mathcal{F},\mathcal{V})\otimes_{\mathcal{O_{X}}}\pi^{*}L^{\otimes -n})\\
    \cong&\Gamma(X,\pi_{*}(\sheafhom_{\mathcal{O_{X}}}(\mathcal{F},\mathcal{V})\otimes_{\mathcal{O_{X}}}\pi^{*}L^{\otimes -n}))\\
    \cong&\Gamma(X,\pi_{*}\sheafhom_{\mathcal{O_{X}}}(\mathcal{F},\mathcal{V})\otimes_{\mathcal{O}_{X}}L^{\otimes -n}).
\end{align*}
Note that $\sheafhom_{\mathcal{O_{X}}}(\mathcal{F},\mathcal{V})$ is torsion-free. Since $\pi$ surjective, $\pi_{*}\sheafhom_{\mathcal{O_{X}}}(\mathcal{F},\mathcal{V})$ is also torsion-free. Since $\mathcal{X}$ is smooth and irreducible, its coarse space $X$ is integral. Applying the lemma of Enriques--Severi--Zariski \cite[\href{https://stacks.math.columbia.edu/tag/0FD7}{Tag 0FD7}]{stacks-project} to $\pi_{*}\sheafhom_{\mathcal{O_{X}}}(\mathcal{F},\mathcal{V})$ on $X$ finishes the proof.
\end{proof}

\begin{proposition}\label{3.6}
Let $R$ be a noetherian ring. Let $\mathcal{X}$ be a smooth separated irreducible tame algebraic stack over $R$ with coarse moduli space $\pi\colon\mathcal{X}\to X$. Let $\mathcal{V}$ be a generating vector bundle on $\mathcal{X}$ and $L$ an ample line bundle on $X$. If $X$ is proper, then the sequence $S=\{\mathcal{V}\otimes_{\mathcal{O_{X}}}\pi^{*}L^{\otimes -n}\}_{n\in\mathbb{Z}}$ is ample in $\operatorname{Coh}(\mathcal{X})$. 
\end{proposition}

\begin{proof}
This follows from Lemmas \ref{3.3}, \ref{3.4} and \ref{3.5}.
\end{proof}

We conclude this section with the following consequence of Lemma \ref{3.4}, which will be useful in the next section.

\begin{lemma}\label{3.7}
    Let $R$ be a noetherian ring. Let $\mathcal{X}$ be a smooth separated tame algebraic stack over $R$ with coarse moduli space $\pi\colon\mathcal{X}\to X$. Let $n\geq 0$ be a fixed integer. Let $\{\mathcal{F}_{i}\}_{1\leq i\leq n}$ be a sequence of coherent sheaves on $\mathcal{X}$. Let $\mathcal{V}$ be a generating vector bundle on $\mathcal{X}$ and let $L$ be an ample line bundle on $X$. Then there exists an integer $N_{n}$, depending on $n$, such that for all $l>N_{n}$, we have $H^{p}(\mathcal{X},\mathcal{F}_{i}\otimes\mathcal{V}\otimes\pi^{*}L^{\otimes l})=0$ for all $p>0$ and $0\leq i\leq n$.
\end{lemma}

\begin{proof}
    By the exactness of $\pi_{*}$ and the projection formula, we have
    \begin{align*}
        H^{p}(\mathcal{X},\mathcal{F}_{i}\otimes\mathcal{V}\otimes\pi^{*}L^{\otimes l})\cong &H^{p}(X,\pi_{*}(\mathcal{F}_{i}\otimes\mathcal{V}\otimes\pi^{*}L^{\otimes l}))\\
        \cong &H^{p}(X,\pi_{*}(\mathcal{F}_{i}\otimes\mathcal{V})\otimes L^{\otimes l})
    \end{align*}
    for all $p>0$ and $0\leq i\leq n$. Therefore, we may assume that $\mathcal{X}$ is a projective scheme over $R$. By Serre vanishing, there exists an integer $m_{i}$ such that for any $l>m_{i}$ we have $H^{p}(\mathcal{X},\mathcal{F}_{i}\otimes\mathcal{V}\otimes\pi^{*}L^{\otimes l})=0$ for all $p>0$. Setting $N_{n}$ to be the maximum element in $\{m_{i}\ |\ 0\leq i\leq n\}$ finishes the proof.
\end{proof}

\section{Producing a Fourier--Mukai kernel}\label{4}

Let $\mathcal{X}$ and $\mathcal{Y}$ be two smooth proper tame algebraic stacks over a noetherian ring $R$. Let $\pi\colon\mathcal{X}\to X$ be the associated coarse moduli space. Let $F\colon D_{\operatorname{\operatorname{Coh}}}^{b}(\mathcal{X})\to D_{\operatorname{Coh}}^{b}(\mathcal{Y})$ be a $R$-linear exact functor such that$$\operatorname{Hom}_{\mathcal{O}_{\mathcal{Y}}}(F(\mathcal{A}),F(\mathcal{B})[i])=0$$for all $\mathcal{A},\mathcal{B}\in\operatorname{Coh}(\mathcal{X})$ and $i<0$.. Assume that $\mathcal{X}$ has the resolution property and $X$ admits an ample line bundle. In this section, we construct a Fourier--Mukai kernel for $F$. By Corollary \ref{2.3}, there exists a resolution for the diagonal of $\mathcal{X}$ as follows:
\begin{equation*}
    \begin{tikzcd}
    \cdots\arrow[r] & \mathcal{E}_{i}\boxtimes\mathcal{G}_{i}\arrow[r] & \cdots\arrow[r] & \mathcal{E}_{0}\boxtimes\mathcal{G}_{0}\arrow[r] & \mathcal{O}_{\Delta_{\mathcal{X}}}\arrow[r] & 0.
    \end{tikzcd}
\end{equation*}
Consider the following complex in $D^{b}_{\operatorname{Coh}}(\mathcal{X}\times_{R}\mathcal{X})$
\begin{equation}\label{eq4.1}
    \begin{tikzcd}
    \cdots\arrow[r] & \mathcal{E}_{i}\boxtimes\mathcal{G}_{i}\arrow[r] & \cdots\arrow[r] & \mathcal{E}_{0}\boxtimes\mathcal{G}_{0}.
    \end{tikzcd}
\end{equation}
For each $n$, we set $M_{n}$ to be the complex
\begin{equation*}
    \begin{tikzcd}
    \mathcal{E}_{n}\boxtimes\mathcal{G}_{n}\arrow[r] & \cdots\arrow[r] & \mathcal{E}_{0}\boxtimes\mathcal{G}_{0},
    \end{tikzcd}
\end{equation*}
which is the $n^{th}$ brutal truncation of the complex in \eqref{eq4.1}. Then we have the following infinite Postnikov system
\begin{equation}\label{eq4.2}
    \begin{tikzcd}
   \cdots & M_{n}\arrow[dr]\arrow[l,dashed,"{[1]}"] & & \cdots\arrow[ll,dashed,"{[1]}"]\arrow[dr] & & M_{0}\arrow[ll,dashed,"{[1]}"]\arrow[dr,"\cong"]\\
    & \cdots\arrow[r] & \mathcal{E}_{n}\boxtimes\mathcal{G}_{n}\arrow[ur]\arrow[rr] & & \cdots\arrow[ur]\arrow[rr]& & \mathcal{E}_{0}\boxtimes\mathcal{G}_{0},
    \end{tikzcd}
\end{equation}
where the dashed arrow represents the shift in $D^{b}_{\operatorname{Coh}}(\mathcal{X}\times_{R}\mathcal{X})$. Set$$\mathcal{F}_{n}=\operatorname{Ker}(\mathcal{E}_{n}\boxtimes\mathcal{G}_{n}\to\mathcal{E}_{n-1}\boxtimes\mathcal{G}_{n-1})$$for all $n>0$. Then we have$$H^{i}(M_{n})=
\begin{cases}
\mathcal{O}_{\Delta_{\mathcal{X}}} & i=0\\
\mathcal{F}_{n} & i=-n\\
0 & i\neq 0,-n.
\end{cases}$$

For a quasicompact and quasiseparated algebraic stack $\mathcal{W}$, we write $\Cohdim(\mathcal{W})$ for the cohomological dimension of $\mathcal{W}$, which is the smallest integer $n_{0}$ such that $H^{n}(\mathcal{W},\mathcal{F})=0$ for every quasicoherent sheaf $\mathcal{F}$ on $\mathcal{W}$ and $n>n_{0}$. In our case, both $\mathcal{X}$ and $\mathcal{Y}$ are smooth and tame. It follows from \cite[Theorem\ B and C]{HR15} that $\mathcal{X}$ and $\mathcal{Y}$ have finite cohomological dimensions. In particular, so do $\mathcal{X}\times_{R}\mathcal{X}$ and $\mathcal{X}\times_{R}\mathcal{Y}$. If $n>\Cohdim(\mathcal{X}\times_{R}\mathcal{X})$, we see that $M_{n}\cong\mathcal{O}_{\Delta_{\mathcal{X}}}\oplus\mathcal{F}_{n}[n]$ as complexes in $D^{b}_{\operatorname{Coh}}(\mathcal{X}\times_{R}\mathcal{X})$ and we obtain a right convolution$$(\delta_{0},0)\colon\mathcal{E}_{0}\boxtimes\mathcal{G}_{0}\longrightarrow\mathcal{O}_{\Delta_{\mathcal{X}}}\oplus\mathcal{F}_{n}[n].$$By the K\"unneth formula, we have
\begin{align*}
    \operatorname{Ext}_{\mathcal{O_{X\times X}}}^{q}(\mathcal{E}_{i}\boxtimes \mathcal{G}_{i},\mathcal{E}_{j}\boxtimes \mathcal{G}_{j})&\cong\bigoplus_{p\in\mathbb{Z}}\operatorname{Ext}_{\mathcal{O_{X}}}^{q+p}(\mathcal{E}_{i},\mathcal{E}_{j})\otimes_{R}\operatorname{Ext}_{\mathcal{O_{Y}}}^{-p}(\mathcal{G}_{i},\mathcal{G}_{j})\\
    &\overset{id\otimes F}{\longrightarrow}\bigoplus_{p\in\mathbb{Z}}\operatorname{Ext}_{\mathcal{O_{X}}}^{q+p}(\mathcal{E}_{i},\mathcal{E}_{j})\otimes_{R}\operatorname{Ext}_{\mathcal{O_{Y}}}^{-p}(F(\mathcal{G}_{i}),F(\mathcal{G}_{j}))\\
    &\cong\operatorname{Ext}_{\mathcal{O}_{\mathcal{X}\times_{R}\mathcal{Y}}}^{q}(\mathcal{E}_{i}\boxtimes F(\mathcal{G}_{i}),\mathcal{E}_{j}\boxtimes F(\mathcal{G}_{j})).
\end{align*}
Therefore we obtain a complex
\begin{equation}\label{eq4.3}
    \begin{tikzcd}
    \cdots\arrow[r] & \mathcal{E}_{i}\boxtimes F(\mathcal{G}_{i})\arrow[r] & \cdots\arrow[r] & \mathcal{E}_{0}\boxtimes F(\mathcal{G}_{0})
    \end{tikzcd}
\end{equation}
in $D^{b}_{\operatorname{Coh}}(\mathcal{X}\times_{R}\mathcal{Y})$, where the maps are given by applying $id\otimes F$ to the corresponding map in \eqref{eq4.1}. Moreover, we have$$\operatorname{Ext}_{\mathcal{O}_{\mathcal{X}\times_{R}\mathcal{Y}}}^{c}(\mathcal{E}_{i}\boxtimes F(\mathcal{G}_{i}),\mathcal{E}_{j}\boxtimes F(\mathcal{G}_{j}))=0$$for all $c<0$. By \cite[Lemma\ 3.2]{CS07}, there exists an infinite Postnikov system
\begin{equation}\label{eq4.4}
    \begin{tikzcd}
   \ & K_{n}\arrow[dr]\arrow[l,dashed,"{[1]}"] & & \cdots\arrow[ll,dashed,"{[1]}"]\arrow[dr] & & K_{0}\arrow[ll,dashed,"{[1]}"]\arrow[dr,"\cong"]\\
    & \cdots\arrow[r] & \mathcal{E}_{n}\boxtimes F(\mathcal{G}_{n})\arrow[ur]\arrow[rr] & & \cdots\arrow[ur]\arrow[rr]& & \mathcal{E}_{0}\boxtimes F(\mathcal{G}_{0}).
    \end{tikzcd}
\end{equation}
We now work towards obtaining a direct sum decomposition for $K_{n}$ in $D^{b}_{\operatorname{Coh}}(\mathcal{X}\times_{R}\mathcal{Y})$ that is similar to the one for $M_{n}$ in $D^{b}_{\operatorname{Coh}}(\mathcal{X}\times_{R}\mathcal{X})$. We begin with several useful auxiliary results.

\begin{lemma}\label{4.1}
Let $\mathcal{X},\mathcal{Y}$ be smooth tame algebraic stacks of finite type over $R$. Let $K$ be a complex in $D_{coh}^{b}(\mathcal{X})$. Let $\pi\colon\mathcal{X}\to X$ be the associated coarse moduli space. Assume that $\mathcal{X}$ has the resolution property, and $X$ is projective. Let $L$ be an ample line bundle on $X$. For a fixed integer $i$, if $H^{i}(\Phi_{K}^{\mathcal{X}\to\mathcal{Y}}(\mathcal{V}\otimes\pi^{*}\mathcal{L}^{\otimes l}))=0$ for every vector bundle $\mathcal{V}$ on $\mathcal{X}$ and every $l\gg 0$, then $H^{i}(K)=0$.
\end{lemma}

\begin{proof}
    The statement is local on $\mathcal{Y}$. So we may assume that $\mathcal{Y}$ is a noetherian affine scheme. Since $\mathcal{X}$ has the resolution property, it is a quotient stack \cite[Theorem\ 1.1]{tot04}. By \cite[Corollary\ 4.5]{DHM22}, it admits a finite flat covering $c_{1}\colon U\to\mathcal{X}$ where $U$ is projective over $\mathcal{Y}$. Let $\tilde{c}_{1}\colon U\to X$ be the composition. We see that $\tilde{c}_{1}$ is quasifinite and proper and thus finite. It follows that $\tilde{c}_{1}^{*}L$ is ample on $U$. Now consider the following commutative diagram.
    \begin{equation*}
        \begin{tikzcd}
            U\arrow[d,"c_{1}",swap] & U\times\mathcal{Y}\arrow[l,"\tilde{q}",swap]\arrow[d,"c",swap]\arrow[rd,"\tilde{p}"] & \\
            \mathcal{X} & \mathcal{X}\times\mathcal{Y}\arrow[l,"q"]\arrow[r,"p",swap] & \mathcal{Y}.
        \end{tikzcd}
    \end{equation*}
    Let $\mathcal{G}$ be a vector bundle on $U$. By the projection formula, we have$$c_{1,*}(\mathcal{G}\otimes\tilde{c}_{1}^{*}L^{\otimes l})\cong c_{1,*}\mathcal{G}\otimes\pi^{*}L^{\otimes l}.$$Note that $c$ is also flat, we see that $c_{1,*}\mathcal{G}$ is a vector bundle on $\mathcal{X}$. By assumption, we have $H^{i}(\Phi_{K}^{\mathcal{X}\to\mathcal{Y}}(c_{1,*}\mathcal{G}\otimes\tilde{c}_{1}^{*}\mathcal{L}^{\otimes l}))=0$ for every $j\gg 0$. Using flat base change and the projection formula, we obtain the following
    \begin{align*}
        \Phi_{K}^{\mathcal{X}\to\mathcal{Y}}(c_{1,*}\mathcal{G}\otimes\pi^{*}\mathcal{L}^{\otimes l})&\simeq Rp_{*}(q^{*}c_{1,*}\mathcal{G}\otimes q^{*}\pi^{*}L^{\otimes l}\otimes^{\mathbb{L}}K)\\
        &\simeq Rp_{*}(c_{*}\tilde{q}^{*}\mathcal{G}\otimes q^{*}\pi^{*}L^{\otimes l}\otimes^{\mathbb{L}}K)\\
        &\simeq Rp_{*}(c_{*}(\tilde{q}^{*}\mathcal{G}\otimes c^{*}q^{*}\pi^{*}L^{\otimes l}\otimes^{\mathbb{L}}Lc^{*}K))\\
        &\simeq Rp_{*}(c_{*}(\tilde{q}^{*}\mathcal{G}\otimes \tilde{q}^{*}c_{1}^{*}\pi^{*}L^{\otimes l}\otimes^{\mathbb{L}}Lc^{*}K))\\
        &\simeq R\tilde{p}_{*}(\tilde{q}^{*}(\mathcal{G}\otimes \tilde{c}_{1}^{*}L^{\otimes l})\otimes^{\mathbb{L}}Lc^{*}K)\\
        &\simeq\Phi_{Lc^{*}K}^{U\to\mathcal{Y}}(\mathcal{G}\otimes\tilde{c}_{1,*}\mathcal{L}^{\otimes l}).
    \end{align*}
    Since $U$ is projective, we may apply the untwisted version of \cite[Lemma\ 2.5]{CS07} to conclude that $H^{i}(Lc^{*}K)=0$. But $c$ is faithfully flat, so we have $H^{i}(K)=0$ as desired. 
\end{proof}

\begin{lemma}\label{4.2}
Let $\mathcal{X},\mathcal{Y}$ be smooth tame algebraic stacks of finite type over $R$. Let $\pi\colon\mathcal{X}\to X$ be the associated coarse moduli space. Let $F\colon D^{b}_{coh}(\mathcal{X})\to D^{b}_{coh}(\mathcal{Y})$ be a $k$-linear exact functor such that $\operatorname{Hom}_{\mathcal{O}_{\mathcal{Y}}}(F(\mathcal{A}),F(\mathcal{B})[i])=0$ for all $\mathcal{A},\mathcal{B}\in\operatorname{Coh}(\mathcal{X})$ and $i<0$. Suppose $\mathcal{X}$ admits a generating vector bundle $\mathcal{V}$ and $X$ is projective. Then for every fixed integer $n>\Cohdim(\mathcal{X}\times_{R}\mathcal{X})$, there exists a positive integer $N_{n}$, depending on $n$, and a coherent sheaf $\mathcal{K}_{n}$ such that$$\Phi_{K_{n}}^{\mathcal{X\to Y}}(\mathcal{V}\otimes\pi^{*}L^{\otimes l})\simeq F(\mathcal{V}\otimes\pi^{*}L^{\otimes l})\oplus F(\mathcal{K}_{n})[n]$$for every $l>N_{n}$.
\end{lemma}

\begin{proof}
    Consider the exact functor
\begin{align*}
    \Psi\colon D^{b}_{\operatorname{Coh}}(\mathcal{X}\times_{R}\mathcal{Y})&\longrightarrow D^{b}_{\operatorname{Coh}}(\mathcal{Y})\\
    N&\longmapsto\Phi_{N}^{\mathcal{X\to Y}}(\mathcal{V}\otimes\pi^{*}L^{\otimes l})=Rp_{*}(Lq^{*}\mathcal{V}\otimes\pi^{*}L^{\otimes l}\otimes_{\mathcal{O}_{\mathcal{X}\times_{R}\mathcal{Y}}}^{\mathbb{L}}N).
\end{align*}
Applying the functor $\Psi$ to \eqref{eq4.4} gives us a Postnikov system for the complex
\begin{equation}\label{eq4.5}
    \begin{tikzcd}
    \cdots\arrow[r] & \Phi_{\mathcal{E}_{n}\boxtimes F(\mathcal{G}_{n})}^{\mathcal{X\to Y}}(\mathcal{V}\otimes\pi^{*}L^{\otimes l})\arrow[r] & \cdots\arrow[r] & \Phi_{\mathcal{E}_{0}\boxtimes F(\mathcal{G}_{0})}^{\mathcal{X\to Y}}(\mathcal{V}\otimes\pi^{*}L^{\otimes l}).
    \end{tikzcd}
\end{equation}
For any $0\leq k\leq n$, we have
\begin{align*}
    \Phi_{\mathcal{E}_{k}\boxtimes F(\mathcal{G}_{k})}^{\mathcal{X\to Y}}(\mathcal{V}\otimes\pi^{*}L^{\otimes l})&=Rp_{*}(Lq^{*}(\mathcal{V}\otimes\pi^{*}L^{\otimes l})\otimes_{\mathcal{O}_{\mathcal{X}\times_{R}\mathcal{Y}}}^{\mathbb{L}}(\mathcal{E}_{k}\boxtimes F(\mathcal{G}_{k})))\\
    &\simeq Rp_{*}(Lq^{*}(\mathcal{V}\otimes\pi^{*}L^{\otimes l}\otimes\mathcal{E}_{k})\otimes_{\mathcal{O}_{\mathcal{X}\times_{R}\mathcal{Y}}}^{\mathbb{L}}Lp^{*}F(\mathcal{G}_{k}))\\
    &\simeq Rp_{*}Lq^{*}(\mathcal{V}\otimes\pi^{*}L^{\otimes l}\otimes\mathcal{E}_{k})\otimes_{\mathcal{O}_{\mathcal{X}\times_{R}\mathcal{Y}}}^{\mathbb{L}}F(\mathcal{G}_{k}).
\end{align*}
By Lemma \ref{3.7}, there exists an integer $N_{n}$ such that for every $l>N_{n}$, we have$$R^{i}p_{*}Lq^{*}(\mathcal{V}\otimes\pi^{*}L^{\otimes l}\otimes\mathcal{E}_{k})=0$$for every $i>0$ and $0\leq k\leq n$. It follows that$$\Phi_{\mathcal{E}_{k}\boxtimes F(\mathcal{G}_{k})}^{\mathcal{X\to Y}}(\mathcal{V}\otimes\pi^{*}L^{\otimes l})\simeq\Gamma(\mathcal{X}, \mathcal{V}\otimes\pi^{*}L^{\otimes l}\otimes\mathcal{E}_{k})\otimes_{R}F(\mathcal{G}_{k})$$for every $i>0$ and $0\leq k\leq n$. On the other hand, consider the exact functor
\begin{align*}
    \Psi^\prime\colon D^{b}_{\operatorname{Coh}}(\mathcal{X\times X})&\longrightarrow D^{b}_{\operatorname{Coh}}(\mathcal{Y})\\
    N&\longmapsto F(\Phi_{N}^{\mathcal{X}\to \mathcal{X}}(\mathcal{V}\otimes\pi^{*}L^{\otimes l})).
\end{align*}
Applying the functor $\Psi^\prime$ to \eqref{eq4.2} gives us a Postnikov system for the complex
\begin{equation}\label{eq4.6}
    \begin{tikzcd}
    \cdots\arrow[r] & F(\Phi_{\mathcal{E}_{n}\boxtimes\mathcal{G}_{n}}^{\mathcal{X}\to \mathcal{X}}(\mathcal{V}\otimes\pi^{*}L^{\otimes l}))\arrow[r] & \cdots\arrow[r] & F(\Phi_{\mathcal{E}_{0}\boxtimes\mathcal{G}_{0}}^{\mathcal{X}\to \mathcal{X}}(\mathcal{V}\otimes\pi^{*}L^{\otimes l})).
    \end{tikzcd}
\end{equation}
Similarly for any $0\leq k\leq n$, we have
\begin{align*}
    F(\Phi_{\mathcal{E}_{k}\boxtimes\mathcal{G}_{k}}^{\mathcal{X}\to \mathcal{X}}(\mathcal{V}\otimes\pi^{*}L^{\otimes l}))&\simeq F(\Gamma(\mathcal{X},\mathcal{V}\otimes\pi^{*}L^{\otimes l}\otimes\mathcal{E}_{k})\otimes_{R}\mathcal{G}_{k})\\
    &\simeq\Gamma(\mathcal{X},\mathcal{V}\otimes\pi^{*}L^{\otimes l}\otimes\mathcal{E}_{k})\otimes_{R}F(\mathcal{G}_{k}).
\end{align*}
Therefore we see that $\Phi_{K_{n}}^{\mathcal{X\to Y}}(\mathcal{V}\otimes\pi^{*}L^{\otimes l})$ and $F(\Phi_{M_{n}}^{\mathcal{X}\to \mathcal{X}}(\mathcal{V}\otimes\pi^{*}L^{\otimes l}))$ are two right convolutions to the same complex in $D^{b}_{\operatorname{Coh}}(\mathcal{Y})$. Moreover, we have$$\operatorname{Ext}_{\mathcal{O_{Y}}}^{c}(F(\Gamma(\mathcal{X},\mathcal{V}\otimes\pi^{*}L^{\otimes l}\otimes\mathcal{E}_{i})\otimes_{R}\mathcal{G}_{i}),F(\Gamma(\mathcal{X},\mathcal{V}\otimes\pi^{*}L^{\otimes l}\otimes\mathcal{E}_{l})\otimes_{R}\mathcal{G}_{l}))=0$$for all $c<0$. By \cite[Lemma\ 3.2]{CS07}, we conclude that$$\Phi_{K_{n}}^{\mathcal{X\to Y}}(\mathcal{V}\otimes\pi^{*}L^{\otimes l})\simeq F(\Phi_{M_{n}}^{\mathcal{X}\to \mathcal{X}}(\mathcal{V}\otimes\pi^{*}L^{\otimes l}))$$for every $l>N_{n}$. Since $n>\Cohdim(\mathcal{X}\times_{R}\mathcal{X})$, we have $M_{n}\cong\mathcal{O}_{\Delta_{\mathcal{X}}}\oplus\mathcal{F}_{n}[n]$. It follows that$$\Phi_{K_{n}}^{\mathcal{X\to Y}}(\mathcal{V}\otimes\pi^{*}L^{\otimes l}))\simeq F(\mathcal{V}\otimes\pi^{*}L^{\otimes l})\oplus F(\Phi_{\mathcal{F}_{n}}^{\mathcal{X}\to \mathcal{X}}(\mathcal{V}\otimes\pi^{*}L^{\otimes l}))[n].$$This finishes the proof.
\end{proof}

Recall that a functor $F\colon D^{b}(\mathcal{X})\to\mathcal{T}$, where $\mathcal{T}$ is a  triangulated category, is {\it bounded} if there exists an integer $m>0$ such that $F(A)$ is concentrated in $[-m,m]$ for any coherent sheaf $A$ on $\mathcal{X}$.

\begin{lemma}\label{4.3}
In the situation of Lemma \ref{4.2}, suppose that $F$ is bounded. Then there exists an integer $n$ and a complex $\overline{K}_{n}$, depending on $n$, such that$$\Phi_{\overline{K}_{n}}^{\mathcal{X\to Y}}(\mathcal{V}\otimes\pi^{*}L^{\otimes l})\simeq F(\mathcal{V}\otimes\pi^{*}L^{\otimes l})$$for every $l>N_{n}$.
\end{lemma}

\begin{proof}
Since $F$ is bounded, it follows from Lemma \ref{4.2} that $\Phi_{K_{n}}^{\mathcal{X\to Y}}(\mathcal{V}\otimes\pi^{*}L^{\otimes l})$ is concentrated in $[-m,m]\cup[-m-n,m-n]$ if $n>\Cohdim(\mathcal{X}\times_{R}\mathcal{X})$. By Lemma \ref{4.1}, we see that $K_{n}$ is concentrated in $[-m,m]\cup[-m-n,m-n]$ as well. Choose an integer $n>\max\{\Cohdim(\mathcal{X}\times_{R}\mathcal{Y})+2m,\  \Cohdim(\mathcal{X}\times_{R}\mathcal{X})\}$. By construction, the intervals $[-m,m]$ and $[-m-n,m-n]$ are disjoint. In this case, we have a direct sum decomposition $K_{n}\simeq K\oplus C_{n}$ where $K$ is concentrated in $[-m,m]$ and $C_{n}$ is concentrated in $[-m-n,m-n]$. It follows that$$\Phi_{K_{n}}^{\mathcal{X\to Y}}(\mathcal{V}\otimes\pi^{*}L^{\otimes l}))\simeq\Phi_{\overline{K}_{n}}^{\mathcal{X\to Y}}(\mathcal{V}\otimes\pi^{*}L^{\otimes l}))\oplus\Phi_{C_{n}}^{\mathcal{X\to Y}}(\mathcal{V}\otimes\pi^{*}L^{\otimes l}))$$for any $l>N$. By construction, the first summand is concentrated in $[-m,m]$ and the second summand is concentrated in $[-m-n,m-n]$. Comparing this decomposition with the one obtained by Lemma \ref{4.2} gives us $\Phi_{\overline{K}_{n}}^{\mathcal{X\to Y}}(\mathcal{V}\otimes\pi^{*}L^{\otimes l}))\simeq F(\mathcal{V}\otimes\pi^{*}L^{\otimes l}))$ for every $l>N_{n}$ as desired.
\end{proof}

We now show that for every fixed integer $n>\max\{\Cohdim(\mathcal{X}\times_{R}\mathcal{Y})+2m,\  \Cohdim(\mathcal{X}\times_{R}\mathcal{X})\}$, the complex $\overline{K}_{n}$ obtained in Lemma \ref{4.3} is a Fourier--Mukai kernel of $F$ with some mild hypothesis. By Lemma \ref{4.2} and \ref{4.3}, there exists an integer $N_{n}$ such that$$\Phi_{\overline{K}_{n}}^{\mathcal{X\to Y}}(\mathcal{V}\otimes\pi^{*}L^{\otimes l})\simeq F(\mathcal{V}\otimes\pi^{*}L^{\otimes l})$$for every $l>N_{n}$. Let $\mathbf{S}_{0}$ be the full subcategory of $\operatorname{Coh}(\mathcal{X})$ whose objects are of the form $\mathcal{V}\otimes\pi^{*}L^{\otimes l}$ where $l>N_{n}$.

\begin{lemma}\label{4.4}
    Let $\mathcal{X},\mathcal{Y}$ be smooth tame algebraic stacks of finite type over $R$. Let $\pi\colon\mathcal{X}\to X$ be the associated coarse moduli space. Let $F\colon D^{b}_{coh}(\mathcal{X})\to D^{b}_{coh}(\mathcal{Y})$ be a $k$-linear exact functor such that $\operatorname{Hom}_{\mathcal{O}_{\mathcal{Y}}}(F(\mathcal{A}),F(\mathcal{B})[i])=0$ for all $\mathcal{A},\mathcal{B}\in\operatorname{Coh}(\mathcal{X})$ and $i<0$. Suppose $\mathcal{X}$ admits a generating vector bundle $\mathcal{V}$ and $X$ is projective. Then for every fixed integer $n>\max\{\Cohdim(\mathcal{X}\times_{R}\mathcal{Y})+2m,\  \Cohdim(\mathcal{X}\times_{R}\mathcal{X})\}$ there exists a bounded complex $\overline{K}_{n}\in D_{\operatorname{Coh}}^{b}(\mathcal{X}\times_{R}\mathcal{Y})$ such that $F|_{\mathbf{S}_{0}}\cong\Phi_{\overline{K}_{n}}|_{\mathbf{S}_{0}}$.
\end{lemma}

\begin{proof}
Let $\mathcal{A}$ be an object in $\mathbf{S}_{0}$. We show that the isomorphism obtained in Lemma \ref{4.3} is unique and natural. To this end, we have a right convolution$$(\delta_{0,\mathcal{A}}^{\prime},0)\colon\Gamma(\mathcal{X},\mathcal{A}\otimes\mathcal{E}_{0})\otimes_{R}F(\mathcal{G}_{0})\to \Phi_{\overline{K}_{n}}(\mathcal{A})\oplus\Phi_{C_{n}}(\mathcal{A})[n]$$for a fixed $n\gg 0$. On the other hand, we also have another right convolution for the same Postnikov system$$(F(\Phi_{\delta_{0,\mathcal{A}}}^{\mathcal{X}\to\mathcal{X}}),0)\colon\Gamma(\mathcal{X},\mathcal{A}\otimes\mathcal{E}_{0})\otimes_{R}F(\mathcal{G}_{0})\to F(\mathcal{A})\oplus F(\mathcal{F}_{n})[n].$$By \cite[Lemma\ 3.3]{CS07}, we get a unique isomorphism $\phi(\mathcal{A})\colon\Phi_{\overline{K}_{n}}(\mathcal{A})\cong F(\mathcal{A})$ such that $F(\Phi_{\delta_{0,\mathcal{A}}}^{\mathcal{X}\to\mathcal{X}})=\phi(\mathcal{A})\circ\delta_{0,\mathcal{A}}^{\prime}$ for all $\mathcal{A}\in\mathbf{S}_{0}$. To see this isomorphism is natural, let $h\colon\mathcal{A}\to\mathcal{B}$ be a morphism in $\mathbf{S}_{0}$. It is not hard to see that both $F(h)\circ\phi(\mathcal{A})$ and $\phi(\mathcal{B})\circ\Phi_{\overline{K}_{n}}(h)$ would make the following diagram commute
\begin{equation*}
\begin{tikzcd}[sep=large]
\Gamma(\mathcal{X},\mathcal{A}\otimes\mathcal{E}_{0})\otimes_{R}F(\mathcal{G}_{0})\arrow[r,"\delta_{0,\mathcal{A}}^{\prime}"]\arrow[d,"\Gamma(h\otimes id)\otimes id"] & \Phi_{\overline{K}_{n}}(\mathcal{A})\arrow[d]\\
\Gamma(\mathcal{X},\mathcal{B}\otimes\mathcal{E}_{0})\otimes_{R}F(\mathcal{G}_{0})\arrow[r,"F(\delta_{0,\mathcal{B}})"] & F(\mathcal{A}).
\end{tikzcd}
\end{equation*}
By \cite[Lemma\ 3.3]{CS07}, we must have $F(h)\circ\phi(\mathcal{A})=\phi(\mathcal{B})\circ\Phi_{\overline{K}_{n}}(h)$. Therefore we see that the ismorphisms $\phi(\mathcal{A})$'s extend to a natural isomorphism $\Phi_{\overline{K}_{n}}|_{\mathbf{S}_{0}}\to F|_{\mathbf{S}_{0}}$.
\end{proof}

Recall that we the sequence $S=\{\mathcal{V}\otimes_{\mathcal{O_{X}}}\pi^{*}L^{\otimes -l}\}_{l\in\mathbb{Z}}$ is ample by Lemma \ref{3.6}. Let $\mathbf{S}$ be the full subcategory of $D^{b}_{\operatorname{Coh}}(\mathcal{X})$ whose objects are elements in $S$. The next step is to show that $F\cong\Phi_{\overline{K}_{n}}$ when restricted to $\mathbf{S}$.

\begin{proposition}\label{4.5}
    In this situation of Lemma \ref{4.4}, there exists a bounded complex $K\in D_{\operatorname{Coh}}^{b}(\mathcal{X}\times_{R}\mathcal{Y})$ such that $F|_{\mathbf{S}}\cong\Phi_{\overline{K}_{n}}|_{\mathbf{S}}$.
\end{proposition}

\begin{proof}
We argue by induction on $l$ as in the proof of \cite[Lemma\ 6.4]{Kaw04}. By Lemma \ref{4.3}, there exists an integer $N_{\mathbf{S}}$ such that $\mathcal{V}\otimes_{\mathcal{O_{X}}}\pi^{*}L^{\otimes l}\in\mathbf{S}_{0}$ for all $l>N_{\mathbf{S}}$. We have shown in Lemma \ref{4.4} that there is an isomorphism$$\phi(l)\colon\Phi_{\overline{K}_{n}}(\mathcal{V}\otimes_{\mathcal{O_{X}}}\pi^{*}L^{\otimes l})\to F(\mathcal{V}\otimes_{\mathcal{O_{X}}}\pi^{*}L^{\otimes l})$$for all $l>N_{\mathbf{S}}$. This covers the base case. For the inductive step, suppose there is an integer $k$ such that for every $k^\prime\geq k$, we have an  isomorphism$$\phi(k^\prime)\colon \Phi_{\overline{K}_{n}}(\mathcal{V}\otimes_{\mathcal{O_{X}}}\pi^{*}L^{\otimes k^\prime})\to F(\mathcal{V}\otimes_{\mathcal{O_{X}}}\pi^{*}L^{\otimes k^\prime}).$$Since $X$ is projective, we may choose a closed immersion $X\to\mathbb{P}_{R}^{N}$ such that $L=\mathcal{O}_{X}(1)$. Using the Beilision resolution, we obtain the following right resolution of $\mathcal{O}_{\mathcal{X}}$
\begin{equation*}\label{eq4.7}
    \begin{tikzcd}
    0\arrow[r] & \mathcal{O}_{\mathcal{X}}\arrow[r] & V_{N}\otimes\pi^{*}\mathcal{O}_{X}(1)\arrow[r] & \cdots\arrow[r] & V_{0}\otimes\pi^{*}\mathcal{O}_{X}(N+1)\arrow[r] & 0,
    \end{tikzcd}
\end{equation*}
where $V_{i}=H^{N}(\mathbb{P}_{R}^{N},\Omega_{\mathbb{P}_{R}^{N}}^{i}(i-N-1))$ for all $0\leq i\leq N$. Tensoring with $\mathcal{V}\otimes\pi^{*}\mathcal{O}_{X}(k)$ yields a right resolution for $\mathcal{V}\otimes\pi^{*}\mathcal{O}_{X}(k)$ as follows
\begin{equation*}\label{eq4.8}
    \begin{tikzcd}[column sep=0.3cm]
    0\arrow[r] & \mathcal{V}\otimes\mathcal{O}_{X}(k)\arrow[r] & V_{N}\otimes\mathcal{V}\otimes\pi^{*}\mathcal{O}_{X}(k+1)\arrow[r] & V_{0}\otimes\mathcal{V}\otimes\pi^{*}\mathcal{O}_{X}(N+k+1)\arrow[r] & 0.
    \end{tikzcd}
\end{equation*}
In other words, $\mathcal{V}\otimes\pi^{*}\mathcal{O}_{X}(k)$ is a left convolution to
\begin{equation*}\label{eq4.9}
    \begin{tikzcd}
    V_{N}\otimes\mathcal{V}\otimes\pi^{*}\mathcal{O}_{X}(k+1)\arrow[r] & \cdots\arrow[r] & V_{0}\otimes\mathcal{V}\otimes\pi^{*}\mathcal{O}_{X}(N+k+1).
    \end{tikzcd}
\end{equation*}
Note that the objects in the complex above are all contained in $\mathbf{S}_{0}$ by assumption. By induction hypothesis, we get the following diagram
\begin{equation*}\label{eq4.10}
    \begin{tikzcd}[column sep=small]
    V_{N}\otimes\Phi_{\overline{K}_{n}}(\mathcal{V}\otimes\pi^{*}\mathcal{O}_{X}(k+1))\arrow[r]\arrow[d,"id\otimes\phi(k+1)"] & \cdots\arrow[r] & V_{0}\otimes\Phi_{\overline{K}_{n}}(\mathcal{V}\otimes\pi^{*}\mathcal{O}_{X}(N+k+1))\arrow[d,"id\otimes\phi(N+k+1)"]\\
    V_{N}\otimes F(\mathcal{V}\otimes\pi^{*}\mathcal{O}_{X}(k+1))\arrow[r] & \cdots\arrow[r] & V_{0}\otimes F(\mathcal{V}\otimes\pi^{*}\mathcal{O}_{X}(N+k+1)),
    \end{tikzcd}
\end{equation*}
where the downward maps are all isomorphisms. We see that $\Phi_{\overline{K}_{n}}(\mathcal{V}\otimes\pi^{*}\mathcal{O}_{X}(k))$ and $F(\mathcal{V}\otimes\pi^{*}\mathcal{O}_{X}(k))$ are left convolutions of complexes that are quasi-isomorphic. By \cite[Lemma\ 3.2]{CS07}, we obtain am ismorphism$$\phi(k)\colon\Phi_{\overline{K}_{n}}(\mathcal{V}\otimes\pi^{*}\mathcal{O}_{X}(k))\to F(\mathcal{V}\otimes\pi^{*}\mathcal{O}_{X}(k)).$$By induction, this is true for all $l\in\mathbb{Z}$. It remains to show that these isomorphisms are natural. To this end, let $h\colon\mathcal{V}\otimes\pi^{*}\mathcal{O}_{X}(k)\to\mathcal{V}\otimes\pi^{*}\mathcal{O}_{X}(k^\prime)$. Then we have the following commutative diagram
\begin{equation*}\label{eq4.11}
    \begin{tikzcd}[column sep=small]
    V_{N}\otimes\Phi_{\overline{K}_{n}}(\mathcal{V}\otimes\pi^{*}\mathcal{O}_{X}(k+1))\arrow[r]\arrow[d,"id\otimes(F(h)\circ\phi(k^\prime+1))"] & \cdots\arrow[r] & V_{0}\otimes\Phi_{\overline{K}_{n}}\mathcal{V}\otimes\pi^{*}\mathcal{O}_{X}(N+k+1))\arrow[d,"id\otimes(F(h)\circ\phi(N+k^\prime+1))"]\\
    V_{N}\otimes F(\mathcal{V}\otimes\pi^{*}\mathcal{O}_{X}(k^\prime+1))\arrow[r] & \cdots\arrow[r] & V_{0}\otimes F(\mathcal{V}\otimes\pi^{*}\mathcal{O}_{X}(N+k^\prime+1)).
    \end{tikzcd}
\end{equation*}
We have two maps from $\Phi_{\overline{K}_{n}}(\mathcal{V}\otimes\pi^{*}\mathcal{O}_{X}(k))$ to $F(\mathcal{V}\otimes\pi^{*}\mathcal{O}_{X}(k^\prime))$, namely $F(h)\circ\phi(k^\prime)$ and $\phi(k)\circ\Phi_{\overline{K}_{n}}(h)$.
By \cite[Lemma\ 3.3]{CS07}, such a map is unique and we have $F(h)\circ\phi(k^\prime)=\phi(k)\circ\Phi_{\overline{K}_{n}}(h)$ as desired.\qedhere
\end{proof}

Proposition \ref{4.5} tells us that the two functors $F$ and $\Phi_{\overline{K}_{n}}$ agree when restricted to an ample sequence in $\operatorname{Coh}(\mathcal{X})$. Therefore, \cite[Proposition\ 3.7]{CS07} applies, and we obtain the following result as an immediate consequence.

\begin{corollary}\label{4.6}
Let $\mathcal{X},\mathcal{Y}$ be smooth tame algebraic stacks of finite type over $R$. Let $\pi\colon\mathcal{X}\to X$ be the associated coarse moduli space. Let $F\colon D^{b}_{coh}(\mathcal{X})\to D^{b}_{coh}(\mathcal{Y})$ be a $k$-linear exact functor such that $\operatorname{Hom}_{\mathcal{O}_{\mathcal{Y}}}(F(\mathcal{A}),F(\mathcal{B})[i])=0$ for all $\mathcal{A},\mathcal{B}\in\operatorname{Coh}(\mathcal{X})$ and $i<0$. Suppose $\mathcal{X}$ admits a generating vector bundle $\mathcal{V}$ and $X$ is projective. If $F$ admits a left adjoint and is bound, then there exists a complex $K\in D_{\operatorname{Coh}}^{b}(\mathcal{X}\times_{R}\mathcal{Y})$ such that $F\cong\Phi_{\overline{K}_{n}}$ for every fixed integer $n>\max\{\Cohdim(\mathcal{X}\times_{R}\mathcal{Y})+2m,\  \Cohdim(\mathcal{X}\times_{R}\mathcal{X})\}$.
\end{corollary}

\section{Proof of Theorem \ref{1.1}}\label{5}

We first show that any exact functor going out of $D_{coh}^{b}(\mathcal{X})$ is bounded.

\begin{lemma}\label{5.1}
Let $\mathcal{X}$ and $\mathcal
Y$ be proper tame algebraic stacks over a noetherian ring $R$. Let $F\colon D_{\operatorname{Coh}}^{b}(\mathcal{X})\to D_{\operatorname{Coh}}^{b}(\mathcal{Y})$ be an $R$-linear exact functor. If $R$ is quasiexcellent and of finite Krull dimension, then $F$ is bounded and admits a left adjoint. \qedhere
\end{lemma}

\begin{proof}
By \cite[Theorem\ B.1]{HP24}, $D_{\operatorname{Coh}}^{b}(\mathcal{X})$ is strongly generated. By \cite[Theorem\ 1.3]{BvdB03}, $D_{\operatorname{Coh}}^{b}(\mathcal{X})$ is saturated. It follows from \cite[\href{https://stacks.math.columbia.edu/tag/0FX8}{Tag 0FX8}]{stacks-project} that $F$ is bounded. By \cite[Corollary\ B.3]{HP24}, $F$ admits a left adjoint.
\end{proof}

Note that Lemma \ref{5.1} is the only statement in this article for which quasiexcellence is necessary as remarked in the introduction. We are now in the position to prove our main theorem.

\begin{proof}[Proof of Theorem \ref{1.1}]
By assumption, $\mathcal{X}$ has the resolution property. Therefore, it is a quotient stack by \cite[Theorem\ 1.1]{tot04} and thus admits a generating vector bundle $\mathcal{V}$. The existence part of the statement then follows from Corollary \ref{4.6} and Lemma \ref{5.1}.

For uniqueness, we follow the standard argument in \cite[Section\ 4.3]{CS07}. Suppose $F\cong\Phi_{K}$ for some bounded complex $K\in D_{\operatorname{Coh}}^{b}(\mathcal{X}\times_{R}\mathcal{Y})$. Applying the Fourier--Mukai transform $\Phi_{K\boxtimes\mathcal{O}_{\Delta_{\mathcal{X}}}}\colon D_{\operatorname{Coh}}^{b}(\mathcal{X}\times_{R} \mathcal{X})\to D_{\operatorname{Coh}}^{b}(\mathcal{X}\times_{R} \mathcal{Y})$ to the Postnikov system in \ref{eq4.1} gives us $\Phi_{K\boxtimes\mathcal{O}_{\Delta}}(M_{n})\simeq\Phi_{K\boxtimes\mathcal{O}_{\Delta_{\mathcal{X}}}}(\mathcal{O}_{\Delta_{\mathcal{X}}})\oplus\Phi_{K\boxtimes\mathcal{O}_{\Delta_{\mathcal{X}}}}(\mathcal{F}_{n})[n]$ for all $n>0$. It is not hard to see that $K\simeq\Phi_{K\boxtimes\mathcal{O}_{\Delta_{\mathcal{X}}}}(\mathcal{O}_{\Delta_{\mathcal{X}}})$. So we have $\Phi_{K\boxtimes\mathcal{O}_{\Delta}}(M_{n})\simeq\Phi_{K\boxtimes\mathcal{O}_{\Delta_{\mathcal{X}}}}(\mathcal{O}_{\Delta_{\mathcal{X}}})\oplus\Phi_{K\boxtimes\mathcal{O}_{\Delta_{\mathcal{X}}}}(\mathcal{F}_{n})[n]$ for all $n>0$. For a fixed $n\gg 0$, we can find an integer $j$ such that $K\simeq\tau^{\geq j}\Phi_{K\boxtimes\mathcal{O}_{\Delta}}(M_{n})$. We may do so because any Fourier--Mukai transform is bounded. Again by computation, we have $\Phi_{K\boxtimes\mathcal{O}_{\Delta_{\mathcal{X}}}}(\mathcal{E}_{i}\boxtimes\mathcal{F}_{i})\simeq\mathcal{E}_{i}\boxtimes F(\mathcal{F}_{i})$ for all $i$. This tells us that $\Phi_{K\boxtimes\mathcal{O}_{\Delta}}(M_{n})$ is the $n^{th}$ convolution of the complex in \ref{eq4.3}. By uniqueness of convolution \cite[Lemma\ 3.2]{CS07}, we have $K\simeq\tau^{\geq j}\Phi_{K\boxtimes\mathcal{O}_{\Delta}}(M_{n})\simeq\tau^{\geq j}K_{n}$ for some fixed numbers $n\gg 0$ and $j$. This proves that $K$ is unique up to quasi-isomorphism.
\end{proof}

In \cite{Cal18}, Calabrese extended the Bondal--Orlov reconstruction theorem to the relative setting under the assumption that the derived equivalence involved is a Fourier--Mukai transform. Theorem \ref{1.1} shows that this assumption is superfluous over a Noetherian ring without appealing to dg-enhancements. We record the following relative Bondal--Orlov theorem as a consequence of the proof of Theorem \ref{1.1}, which we expect to be well-known among experts.

\begin{corollary}\label{5.3}
Let $R$ be a Noetherian ring. Let $X$ and $Y$ be two smooth projective varieties over $R$. If $X$ has ample or antiample canonical bundle over $R$, then $X\cong Y$ as schemes over $R$ if and only if $D_{\operatorname{Coh}}^{b}(X)\cong D_{\operatorname{Coh}}^{b}(Y)$ as $R$-linear triangulated categories. Moreover, if $R$ is connected, then the group of $R$-linear autoequivalences of $D_{\operatorname{Coh}}^{b}(X)$ is given by$$\mathbb{Z}\times(\operatorname{Aut}_{R}(X)\ltimes\operatorname{Pic}(X)).$$
\end{corollary}

\begin{proof}
    By \cite[Proposition\ 2.4]{CS07}, any equivalence $F\colon D_{\operatorname{Coh}}^{b}(X)\to D_{\operatorname{Coh}}^{b}(Y)$ is bounded. Corollary \ref{4.6} shows that $F$ is a Fourier--Mukai transform. The statement then follows from \cite[Theorem\ A\ and\ C]{Cal18}.
\end{proof}

\appendix
\section{tame stacks and the resolution property}

In this appendix, we provide an amplification of \cite[Corollary\ 4.5]{DHM22} and use it to show that any smooth separated generically tame Deligne--Mumford stack with quasiprojective coarse moduli space over a noetherian base has the resolution property. 

\begin{lemma}[cf.\ {\cite[Corollary\ 4.5]{DHM22}}]\label{A.1}
    Let $S$ be a noetherian scheme. Let $\mathcal{X}$ be a quasicompact algebraic stack over $S$ with finite diagonal and $X$ the associated coarse moduli space. Let $\mathcal{Y}$ be an \'etale gerbe over $\mathcal{X}$. If $\mathcal{X}$ is a quotient stack and $X$ admits an ample line bundle, then there exists a scheme $Z$ over $S$ that admits an ample line bundle such that $\mathcal{Y}$ admits a finite flat covering $\mathcal{W}\to\mathcal{Y}$, where $\mathcal{W}$ is an \'etale gerbe over $Z$. Moreover, if $\mathcal{W}$ has the resolution property, so does $\mathcal{Y}$.
\end{lemma}

\begin{proof}
It follows from \cite[Corollary\ 4.5]{DHM22} that $\mathcal{X}$ admits a finite flat covering $Z\to X$, where $Z$ is a scheme admitting an ample line bundle. Then the base change $\mathcal{Y}_{Z}$ is an \'etale gerbe over $Z$ and the projection $\mathcal{Y}_{Z}\to\mathcal{Y}$ is a finite flat covering as desired. The latter statement follows from the former as the resolution property descends along finite flat coverings by \cite[Proposition\ 5.3(\romannum{7})]{gro17}.
\end{proof}

We say an algebraic stack is generically tame if it contains a dense open substack that is tame. 

\begin{corollary}[cf.\ {\cite[Theorem\ 2.2]{KV04}}]\label{A.2}
    Let $S$ be a noetherian scheme. Let $\mathcal{X}$ be a smooth separated Deligne--Mumford stack with quasiprojective coarse moduli space over $S$. If $\mathcal{X}$ is generically tame, then it has the resolution property and thus a quotient stack.
\end{corollary}

\begin{proof}
    First of all, we may assume $\mathcal{X}$ is connected. Let $\pi\colon\mathcal{X}\to X$ be the coarse moduli space associated with $\mathcal{X}$. By \cite[Proposition\ 2.1]{Ols07}, $\pi$ admits a factorization$$\mathcal{X}\overset{\alpha}{\longrightarrow}\overline{\mathcal{X}}\overset{\overline{\pi}}{\longrightarrow}X,$$where $\overline{\mathcal{X}}$ is a smooth separated Deligne--Mumford stack with generically trivial stabilizer, $\alpha$ is a \'etale gerbe, and $\overline{\pi}$ is the coarse moduli space associate with $\overline{\mathcal{X}}$. Since $\mathcal{X}$ is generically tame, it can be arranged that $\overline{\mathcal{X}}$ is tame. By \cite[Theorem\ 2.18]{EHKV01}, we see that $\overline{\mathcal{X}}$ is a quotient stack. By Lemma \ref{A.1}, we may assume that $\mathcal{X}$ is a tame Deligne--Mumford gerbe over $X$. By the proof of \cite[Theorem\ 6.2]{BHM21}, we may assume that $\mathcal{X}$ is banded by the finite constant group scheme associated to$$\mathbb{Z}/p_{1}^{m_{1}}\mathbb{Z}\times_{X}\cdots\times_{X}\mathbb{Z}/p_{n}^{m_{n}}\mathbb{Z}.$$Since $\mathcal{X}$ is tame, we see that $p_{j}$ is invertible in $S$ for all $j$. Again by \cite[Proposition\ 5.3(\romannum{7})]{gro17}, we may replace $S$ by a finite \'etale covering of itself that contains all the $p_{j}^{m_{j}}$th roots of unity for each $j$. In this case, we have $\mathbb{Z}/p_{j}^{m_{j}}\mathbb{Z}\cong\mu_{p_{j}^{m_{j}}}$ as group schemes over $S$. By \cite[Theorem\ 3.6]{EHKV01} and \cite[Proposition\ 5.3(\romannum{4})]{gro17}, we see that $\mathcal{X}$ is a quotient stack. By \cite[Proposition\ 16]{Mat21}, we see that the coarse moduli map $\mathcal{X}\to X$ has the resolution property. Since $X$ is quasiprojective, it also has the resolution property. Hence $\mathcal{X}$ has the resolution property. Since $\mathcal{X}$ is smooth with the resolution property, it is a quotient stack by \cite[Theorem\ 1.1]{tot04}.
\end{proof}

\bibliographystyle{amsalpha}
\bibliography{relative.bib}

\end{document}